\documentclass[11pt]{amsart}

\usepackage{amsthm, amsfonts, amsmath}
\usepackage{dsfont}
\usepackage{centernot}
\usepackage{parskip} 
\usepackage[margin=3.5cm]{geometry}
\usepackage{mathdefs}
\usepackage[percent]{overpic}

\usepackage{tikz}
\usetikzlibrary{snakes}
\newcommand{\SNAKE}[1]{
			\hspace{0.3cm}
			\begin{tikzpicture}
  				\draw[thick, snake] (-1,0) -- (1,0);
				\draw[thick, ->, >=stealth] (1,0) -- (1.5,0);
				\node at (0,-#1) {\phantom{.}};
			\end{tikzpicture}
			\hspace{0.3cm} }

\newcounter{thm}
\newtheorem{theorem}[thm]{Theorem}
\newtheorem{corollary}[thm]{Corollary}

\newtheorem*{conjecture*}{Conjecture}
\newtheorem{lemma}[thm]{Lemma}
\newtheorem{prop}[thm]{Proposition}
\newtheorem*{statement*}{Statement}

\theoremstyle{definition}
\newtheorem{example}[thm]{Example}

\theoremstyle{definition}

\newcommand{\hm}{\cong}

\renewcommand{\S}{\Sigma}
\newcommand{\sol}{\mathit{Sol}}

\newcommand{\mless}{\backslash\backslash}
\newcommand{\SmL}{\Sigma \mless L}

\newcommand{\abs}[1]{|#1|}

\begin{document}

\title{The Simple Loop Conjecture for 3-Manifolds Modeled on $\sol$}
\author{Drew Zemke}
\date{}
\maketitle

\begin{abstract}
	The simple loop conjecture for 3-manifolds states that every 2-sided immersion of a closed surface into a 
	3-manifold is either injective on fundamental groups or admits a compression. This can be viewed as a 
	generalization of the Loop Theorem to immersed surfaces. We prove the conjecture in the case that the 
	target 3-manifold admits a geometric structure modeled on $\sol$.
\end{abstract}

\section{Introduction}
	\label{sec:intro}
	
	The Simple Loop Conjecture for 3-manifolds is as follows.
	
	\begin{conjecture*}[Problem 3.96 in \cite{Kirby1995}]
		Let $\S$ be a closed surface and let $M$ be a closed 3-manifold.  If $F : \S \to M$ is a 
		2-sided immersion for which the induced map $F_* : \pi_1 \S \to \pi_1M$ is not injective,
		then there is an essential simple loop in $\S$ that represents an element of the kernel of $F_*$.
	\end{conjecture*}
	
	When the map $F$ is an embedding, this follows from the loop theorem of Papakyriakopoulos (see, 
	for instance, \cite{Hempel1976}).
	
	The Simple Loop Conjecture is known to hold when the target 3-manifold is a Seifert fibered 3-manifold 
	or a graph 3-manifold, by the work of Hass \cite{Hass1987} and Rubinstein-Wang\footnote{It is unclear whether the techniques of \cite{Rubinstein1998} apply to $\sol$ manifolds, 
	though they seem to be implicitly ruling them out (see for instance, \cite[Lemma 1.0.2]{Rubinstein1998}).
	At any rate, the techniques in this paper offer a substantially different approach to the problem.} \cite{Rubinstein1998}, 
	respectively.  An analogous result for maps between surfaces is due to Gabai \cite{Gabai1985}.  
	
	The goal of this paper is the following result.
	
	\begin{theorem}
		\label{thm:SLCforSolvmflds}
		The Simple Loop Conjecture holds when the target 3-manifold admits a geometric structure modeled on 
		$\sol$.
	\end{theorem} 
	
	If $M$ is a 3-manifold that is finitely covered by a torus bundle over $S^1$, then $M$ admits a geometric
	structure modeled on one of Euclidean 3-space, \textit{Nil}, or $\sol$.  Since all Euclidean and
	\textit{Nil} manifolds are Seifert fibered (see \cite{Scott1983}), we obtain the following corollary.
	
	\begin{corollary}
		\label{cor:SLCforMfldsCoveredByTorusBundles}
		The Simple Loop Conjecture holds when the target 3-manifold is finitely covered by a torus
		bundle over $S^1$.
	\end{corollary}
	
	This document is organized as follows.  In Section \ref{sec:defs} we give some definitions and notation for 
	the objects that will be studied.  Section \ref{sec:solvmflds} contains a brief survey of which compact
	3-manifolds admit geometric structures modeled on $\sol$.  This entails a refinement of a classification
	given by Scott in \cite{Scott1983}, and reduces the problem at hand to studying maps from closed surfaces into 
	certain kinds of torus bundles over $S^1$ and orientable torus semi-bundles.  In 
	Sections \ref{sec:TBs} and \ref{sec:TSBs} we give proofs of the Simple Loop Conjecture for these two types of 
	3-manifold, respectively.  We conclude in Section \ref{sec:SLCforMetabelianGroups} with some remarks regarding 
	how the results presented here relate to
	a group-theoretic formulation of the Simple Loop Conjecture, and it fails to hold when the target group is 
	metabelian.

\subsection*{Acknowledgment}
	The author is extremely grateful to Jason Manning for his thoughtful advice, friendly critique, and 
	patience.  An additional thanks is due to Alan Reid for pointing out the connection between Example 
	\ref{ex:group_slc_counterex} and Casson's construction in \cite{Livingston2000}.

\section{Definitions}
	\label{sec:defs}

	If $M$ is a connected manifold, the \emph{orientation character} of $M$ is a homomorphism 
	$\rho_M : \pi_1 M \to \Z/2$
	whose value on $b \in \pi_1M$ is nontrivial if and only if some (and hence any) loop in $M$ representing
	$b$ is orientation reversing.  (Equivalently, $\rho_M(b)$ is nontrivial if and only if $b$
	acts on the universal cover of $M$ by an orientation reversing homeomorphism.)  A manifold
	is orientable if and only if its orientation character is trivial.  
	
	If $M$ and $N$ are connected manifolds with orientation characters $\rho_M$ and $\rho_N$, a map $F : M \to N$ 
	is called \emph{2-sided} if $\rho_N \circ F_* = \rho_M$.  Otherwise $F$ is \emph{1-sided}.  Hence $F$
	is 2-sided if and only if it takes orientation preserving loops in $M$ to orientation preserving loops in $N$, 
	and likewise for orientation reversing loops.  There are other (equivalent) definitions of 2-sidedness for 
	immersions of manifolds, but since most of the arguments in this paper involve the fundamental groups of the 
	manifolds in question, the given definition will be more useful.  
	
	We will call a loop in a manifold $M$ \textit{essential} if it is neither nullhomotopic nor homotopic into
	the boundary of $M$.  Loops that are not essential will be called \textit{inessential}.
	
	For a space $X$, we write $\abs{X}$ to denote the number of connected components of $X$.  For
	a compact surface $\S$ with $L \subset \S$ an embedded closed 1-manifold, we will write 
	$\SmL$ to denote the metric completion of $\Sigma \less L$ (with respect to some choice of complete metric
	on $\S$).  Thus $\SmL$ is the space
	obtained by gluing copies of $S^1$ onto the open ends of $\Sigma \less L$.
		
	We refer the reader to \cite{Scott1983} for an explanation of what it means for a manifold to admit
	a geometric structure, as well as some basic facts about the Euclidean, \textit{Nil}, and $\sol$ geometries.  
	In particular, we will need the following two results.

	\begin{theorem}[{\cite[Theorem 5.2]{Scott1983}}]
	\label{thm:ScottUniqueGeoms}
	If $M$ is a closed 3-manifold which admits a geometric structure modeled on one of the eight geometries,
	then the geometry involved is unique.
	\end{theorem}
	
	\begin{corollary}[see {\cite[Theorem 5.3(ii)]{Scott1983}}]
	\label{thm:ScottSeifertMflds}
	If $M$ is a closed 3-manifold that admits a Seifert fibering, then $M$ does not admit a geometric structure
	modeled on $\sol$.
	\end{corollary}

\subsection{Torus Bundles and Semi-Bundles}
	\label{sec:TSBdefs}

	By \emph{torus bundle} we mean a fiber bundle over $S^1$ whose fibers are tori.  This can also be viewed 
	as a quotient $T\times I / ((p,0) \sim (\phi(p),1))$ where $T$ is a torus and $\phi : T \to T$ is a 
	homeomorphism.  
	
	For each $i \in \{1,2\}$, let $N_i$ be either a twisted $I$-bundle over a torus or a Klein bottle, so that 
	$\bd N_i \hm T$.  A \emph{torus semi-bundle} $M = N_1 \cup_\phi N_2$ is obtained by gluing 
	$N_1$ and $N_2$ by a homeomorphism $\phi : \bd N_1 \to \bd N_2$.  Such a 3-manifold is orientable 
	if and only if both $N_1$ and $N_2$ are twisted $I$-bundles over a Klein bottle.
	
	If $M$ is a torus semi-bundle, at times we will refer to the \emph{middle torus} of $M$, which is the image
	of $\bd N_1$ and $\bd N_2$ after the gluing.  We will also make use of maps $\rho_i : \pi_1 N_i \to \Z/2$,
	which are the quotients of $\pi_1 N_i$ by the index two subgroup corresponding to the double covers of $N_i$ 
	by the product $T \times I$.  (This is sometimes called the \emph{monodromy} of the $I$-bundle $N_i$.)
	Notice that, for $b \in \pi_1 N_i$, $\rho_i(b)$ is trivial if and only if $b$ is 
	represented by a loop that is homotopic into $\bd N_i$.  Furthermore, when $N_i$ is a twisted $I$-bundle over 
	a torus (and is therefore nonorientable), $\rho_i$ coincides with the orientation character of $N_i$.  
	
	If $M$ is a torus semi-bundle, then there is a double cover of $M$ that is the union of the two $T \times I$ 
	double covers of $N_1$ and $N_2$ along their boundaries (via some homeomorphism of the torus).  This is a 
	torus bundle over a circle, and is in turn covered by $T \times \R$ with deck group $\Z$.  Hence $M$ is covered 
	by $T \times \R$ with deck group the \emph{infinite dihedral group} 
	$D = \ang{g_1, g_2 \mid g_1^2 = g_2^2 = 1}$.  
	The induced action
	on $\R$ is the usual discrete action of $D$ on $\R$, where $g_1$ and $g_2$ act by 
	reflections about $0$ and $1$, respectively.  The projection $T \times \R \to \R$ therefore induces a 
	projection $M \to I(2,2)$, where $I(2,2)$ is a 1-dimensional orbifold called the \emph{mirrored interval}.  
	(See \cite{Cooper2000} for definitions and notation.)  It follows that $M$ can be viewed as an \emph{orbifold 
	fiber bundle} over $I(2,2)$.  The generic fibers of this bundle are 2-sided tori in $M$, and the fibers over 
	the mirrored points are the 1-sided tori or Klein bottles of $M$.

\section{Classification of Compact 3-Manifolds Modeled on $\sol$}
	\label{sec:solvmflds}
	
In \cite{Scott1983}, Scott gives the following classification of closed $3$-manifolds modeled on $\sol$.
	(Note that a homeomorphism $\phi : T \to T$ of a torus is called \emph{hyperbolic} if $\phi_*$ acts on 
	$H_1(T;\Z)$ with $\mbox{tr}(T)^2 > 4$.)
	
\begin{theorem}[{\cite[Theorem 5.3(i)]{Scott1983}}]
	\label{thm:ScottSolvmflds}
	Let $M$ be a closed 3-manifold. Then $M$ possesses a geometric structure modeled on $\sol$ if and only if 
	$M$ is a finitely covered by a torus bundle over $S^1$ with hyperbolic monodromy.  In particular, $M$ itself 
	is either a bundle over $S^1$ with fibre the torus or Klein bottle or is the union of two twisted $I$-bundles
	over the torus or Klein bottle.
\end{theorem}

We refine this classification as follows. 

\begin{theorem}
	\label{thm:solvmfldClassification}
	Let $M$ be a closed 3-manifold.  Then $M$ possesses a geometric structure modeled on $\sol$ if and only
	if one of the following holds:
	\begin{enumerate}
		\item $M$ is a torus bundle over $S^1$ with hyperbolic monodromy, or
		\item $M$ is an orientable torus semi-bundle with gluing map (in canonical coordinates) given by
			$\begin{pmatrix} r & s \\ t & u \end{pmatrix}$ where $rstu \ne 0$.
	\end{enumerate}
\end{theorem}
	The notion of \textit{canonical coordinates} on the middle torus of a torus semi-bundle is explained in the 
	definition that precedes Proposition 1.5 of \cite{Sun2010}.  
\begin{proof}
	It is shown in \cite{Sun2010} that an orientable torus semi-bundle admits a $\sol$ structure if and only
	if its gluing map is of the form stated above.  Hence to complete the proof we must show that the other types 
	of 3-manifolds mentioned in Scott's classification do \emph{not} admit geometric structures modeled on $\sol$.  
	
	\textsc{Case 1.} $M$ is a Klein bottle bundle over $S^1$.  Let
	\[
		B = \ang{a,b \mid aba\inv b = 1}
	\]
	be the fundamental group of a Klein bottle, and let $A = \ang{a^2, b} \iso \Z \oplus \Z$ be the 
	normal subgroup of $B$ corresponding to the double cover of the Klein bottle by a torus.  The fundamental group 
	of $M$ has the form
	\[
		\pi_1M = \ang{B, t \mid txt\inv = \phi(x),\; \forall x \in B}
	\]
	for some automorphism $\phi$ of $B$ coming from a homeomorphism of the Klein bottle.  
	
	We now show that every such automorphism of $B$ preserves the subgroup $A$.  We first observe that
	every element of $B$ can be written uniquely as $a^i b^j$ for $i, j \in \Z$.  Since $\phi$ must
	preserve the commutator subgroup $[B,B] = \ang{b^2}$, we have $\phi(b^2) = b^{\pm 2}$, and a short
	computation shows that in fact $\phi(b) = b^{\pm 1}$.  It follows that $\phi(a) = a^i b^j$ where
	$i, j \in \Z$ and $i$ is odd, since otherwise $\phi$ has image in the proper subgroup $A$.  We have
	\[	
		\phi(a^2) = (a^i b^j)(a^i b^j) = (a^i a^i)(b^{-j}b^j) = a^{2i},
	\]
	and similarly $\phi\inv(a^2) = a^{2i'}$ for some $i' \in \Z$.  From 
	$a^2 = \phi\inv(\phi(a^2)) = a^{2i \cdot i'}$ we find that $i \cdot i' = 1$, and so $i = \pm 1$.  In summary, 
	$\phi(b) = b^{\pm 1}$ and $\phi(a^2) = a^{\pm 2}$,
	so $\phi$ preserves the subgroup $A$.  
	
	We therefore conclude that $\pi_1 M$ contains an index-2 subgroup of the form
	\[
		H = \ang{A, t \mid txt\inv = \phi|_A(x),\; \forall x \in A}.
	\]	
	Let $\hat M$ be the double cover of $M$ corresponding to $H$, which is a torus bundle over $S^1$ with 
	monodromy $\phi|_A$.  	
	By the argument in the previous paragraph, there is a choice of basis for $A$ so that 
	\[
		\phi|_A = \begin{pmatrix} \pm 1 & 0 \\ 0 & \pm 1 \end{pmatrix}.
	\]
	Therefore $\phi|_A$ corresponds to a periodic homeomorphism of the torus, and so $\hat M$ admits a Euclidean 
	structure by \cite[Theorem 5.5]{Scott1983}.  It follows that $M$ is not does not admit a $\sol$ structure,
	for if it did the structure could be lifted to a $\sol$ structure on $\hat M$, which would violate 
	Theorem \ref{thm:ScottUniqueGeoms}.
	
	\textsc{Case 2.}  $M$ is a Klein bottle semi-bundle.  Then $M$ is double covered by a Klein bottle bundle over 
	$S^1$ and therefore has a degree-4 cover that is a torus bundle over $S^1$ that admits a Euclidean structure.
	As in the previous case, $M$ does not admit a $\sol$ structure.
	
	\textsc{Case 3.} $M$ is a nonorientable torus semi-bundle.  Then $M$ is the union of two twisted $I$-bundles
	$N_1$ and $N_2$ over a torus or Klein bottle, at least one of which (say $N_1$) is an $I$-bundle over a torus.  
	We will show that $M$ admits a Seifert fibering, and therefore does not admit a $\sol$ 
	structure by Corollary \ref{thm:ScottSeifertMflds}. 
	
	Choose an arbitrary Seifert fibration for $N_2$; up to isomorphism there are precisely two of these 
	when $N_2$ is an $I$-bundle over a Klein bottle (see \cite{Hatcher3mflds}, for instance) and infinitely many 
	when $N_2$ is an $I$-bundle over a torus, as we will show.  
	
	If $T$ is a torus, then for any $p/q \in \Q \cup \{\infty\}$, $T$ can be foliated by $p/q$-curves.  This
	foliation extends to the product Seifert fibration of $T \times I$ by $p/q$-curves in each torus 
	$T \times \{t\}$.  
	Finally, since the covering involution corresponding to the cover $T\times I \to N_1$ preserves the fibration 
	on $T \times I$, it descends to a Seifert fibration of $N_1$ so that $\bd N_1$ is foliated by $p/q$ curves.
	Note that this is the one of the ``generalized'' Seifert fibrations as defined in \cite{Scott1983}, as
	the critical fibers are not isolated.  In fact, the one-sided torus in $N_1$ forms a subsurface of critical 
	fibers.
	
	It follows that a Seifert fibration on $M$ can be constructed by choosing a Seifert fibration on $N_1$ so
	that the foliation of the boundary agrees with the image of the foliation of $\bd N_2$ under the gluing
	map.  
\end{proof}

\section{Torus Bundles}
	\label{sec:TBs}

The first of the two main theorems that will imply Theorem $\ref{thm:SLCforSolvmflds}$ is the following.

\begin{theorem}
	\label{thm:SLCforTBs}
	If $M$ is a torus bundle, then the Simple Loop Conjecture holds for $M$.
\end{theorem}

In fact, a slightly stronger result holds for most surfaces.

\begin{theorem}
	\label{thm:SLCforTBs_strong}
	Let $\S$ be a closed surface and let $M$ be a torus bundle. If $\chi(\S)$ is even and negative
	and $F : \S \to M$ is a 2-sided map, then there is a essential simple loop in $\S$ that represents an element 
	of $\ker F_*$.  If $\chi(\S)$ is odd then there is no $2$-sided map $\S \to M$.
\end{theorem}

After we prove Theorem \ref{thm:SLCforTBs_strong}, to complete the proof of Theorem
	\ref{thm:SLCforTBs} it will remain to handle the two cases where $\chi(\S) = 0$.  The Simple Loop
	Conjecture is known to hold for maps $\S \to M$ where $\S$ is a torus and $M$ is any 3-manifold 
	\cite[Section 4.4]{Hass1987}, and Proposition \ref{prop:SLCforKleinbottles} will deal with the case in which 
	$\S$ is a Klein bottle.

Let $L$ be a (not necessarily connected) 1-submanifold of a surface $\S$ and let $\alpha$ be an arc in $\S$ 
	with endpoints on $L$ and interior disjoint from $L$.  Then \emph{surgery} of $L$ along $\alpha$
	entails fattening $\alpha$ to a strip $I \times I$ with $L \cap (I \times I) = \bd I \times I$, deleting 
	the interior of $\bd I \times I$ from $L$, and gluing in $I \times \bd I$ to $L$.  Notice that
	if $\alpha$ is an arc between two distinct components of $L$, then the result of surgery along $\alpha$ is to 
	connect the two components of $L$ by a bridge, as shown in Figure \ref{fig:loop_join}.
	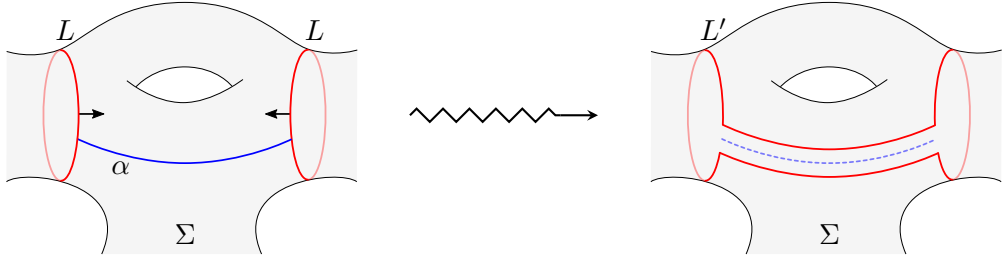
\begin{figure}[t]
		\begin{center}
			\begin{overpic}[width=0.32\textwidth]{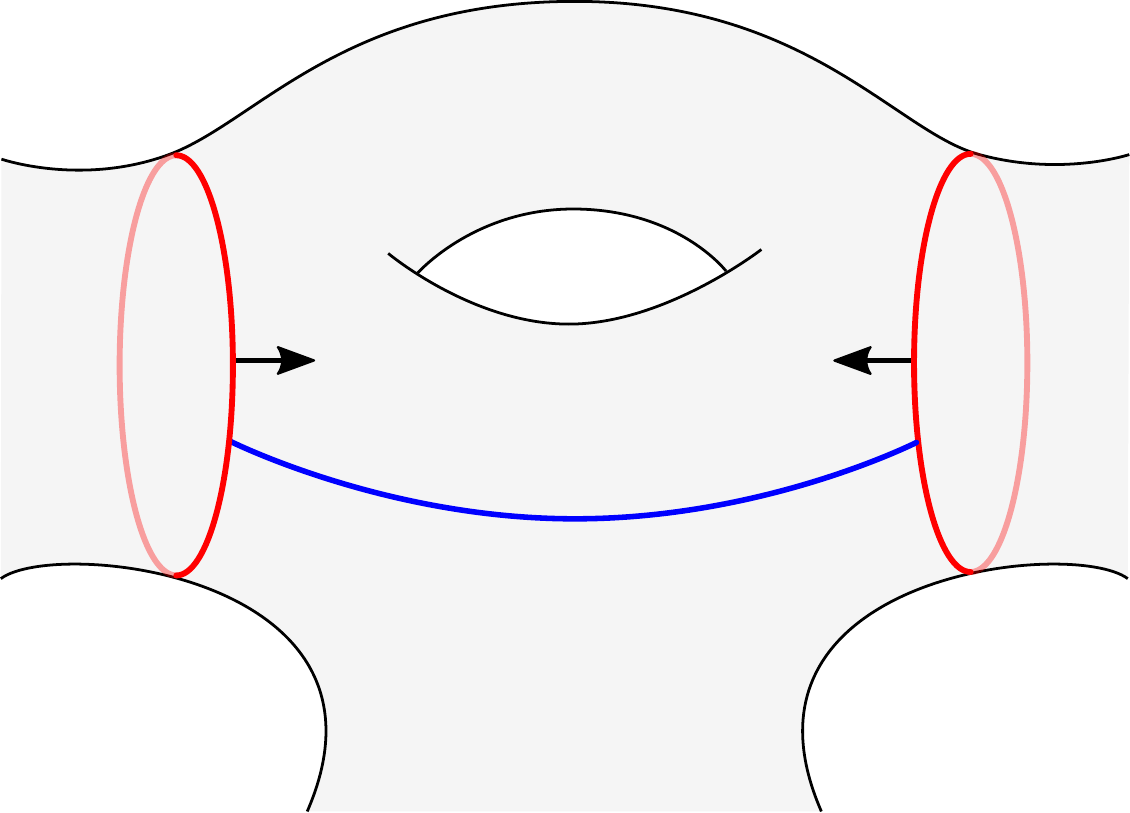}
	 			\put (30,23) {$\alpha$}
				\put (48, 3) {$\S$}
				\put (14,61) {$L$}
				\put (85,61) {$L$}
			\end{overpic}
			\SNAKE{1.7}
			\begin{overpic}[width=0.32\textwidth]{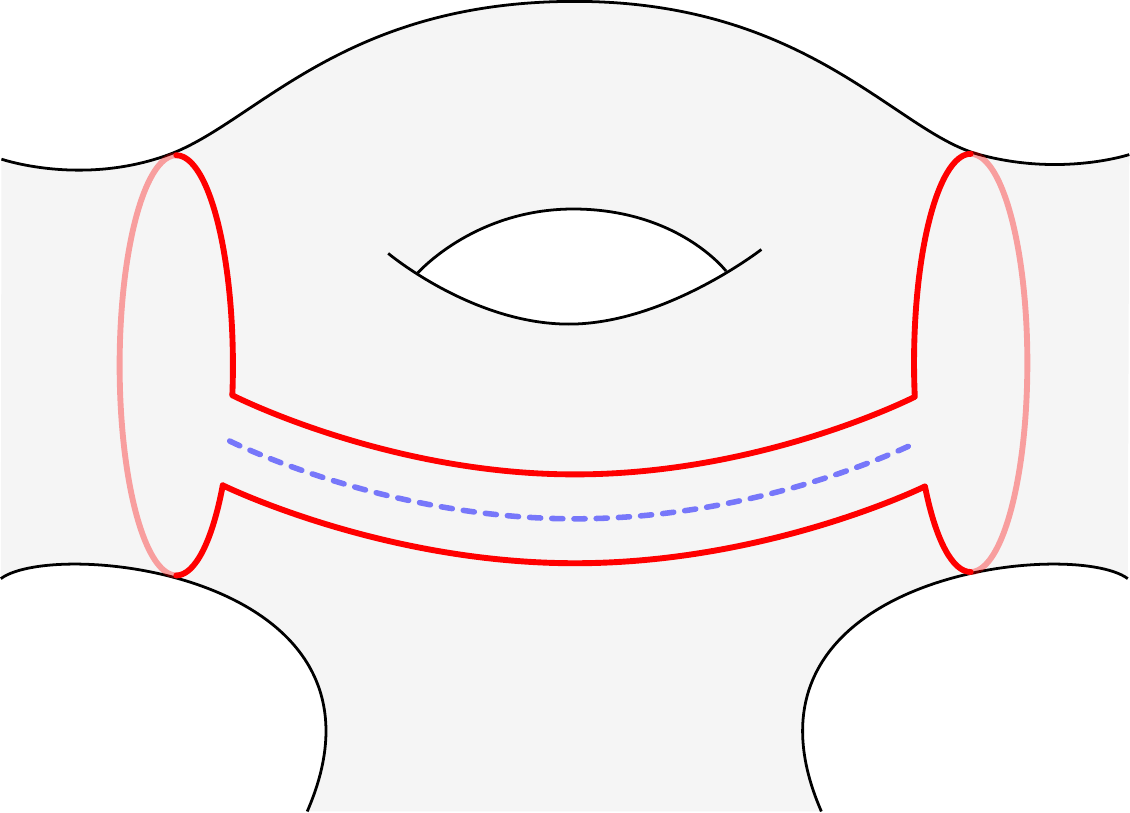}
				\put (48, 3) {$\S$}
				\put (14,61) {$L'$}
			\end{overpic}
		\end{center}
		\caption{Surgery along $\alpha$ reduces the number of components of $L$ by one.}
		\label{fig:loop_join}
	\end{figure}
	
The following can be established by a standard homomtopy argument.

\begin{lemma}
	\label{lem:MapsToInterval}
	Let $\S$ be a (not necessarily closed) surface, 
	let $J$ denote the open interval $(0,1)$, and let $H : \S \to J$ be a map that 
	is transverse to a point $r \in J$.  If $\alpha$ is an arc that connects two components
	of $L = H\inv(r)$ whose interior is disjoint from $L$, then $H$ can be homotoped in a neighborhood of 
	$\alpha$ so that the preimage of $r$ changes by surgery along $\alpha$.
\end{lemma}

\begin{lemma}
	\label{lem:MapsToCircle}
	Let $\S$ be a closed surface, let $G : \S \to S^1$ be a $\pi_1$-surjective map, and choose $q \in S^1$.  
	Then $G$ can be homotoped so that the preimage $L = G\inv(q)$ is a essential 2-sided simple loop in $\S$.
\end{lemma}
\begin{proof}
	Choose $G$ within its homotopy class so that $q$ is a regular value of $G$
	and $L = G\inv(q)$ is a collection of disjoint simple loops in $\S$ with a \emph{minimal} number of 
	components.  Observe that $L$ is 2-sided but may not be connected.  We shall 
	show that the minimality assumption on $L$ along with the assumption that $G$ is $\pi_1$-surjective forces 
	$L$ to be connected.
	
	Choose a co-orientation of  $q \in S^1$ and pull it back to a co-orientation of $L$ in $\S$.
	We summarize this data by drawing a single arrow orthogonal to each component of $L$ 
	that indicates to which side of each component the co-orientation is pointing, as demonstrated in Figures
	\ref{fig:loop_join} and \ref{fig:starfish_surface}.  When we cut $\S$ along $L$ to obtain $\SmL$, we
	label the boundary components of the resulting surface with the co-orientations of the 
	components of the $L$ that the boundary components correspond to.
		
	We can homotope $G$ to reduce the number of components of $L$ whenever a component
	$\S_0$ of $\SmL$ has two boundary loops that are either \emph{both co-oriented into} or 
	\emph{both co-oriented out of} $\S_0$.  This happens, for instance, whenever $\S_0$ has three or more 
	boundary components.  
	Start by choosing a simple arc $\alpha \subset \S_0$ connecting the two boundary components of $\S_0$ with 
	coherent co-orientations, so that $G(\alpha)$ is a nullhomotopic loop in $S^1$ based at $q$.  If $U$ is
	a small neighborhood of $\alpha$ in $\S$, then we can homotope $G$ with support in $U$ so that $G|_U$ is 
	not surjective.  Hence $G|_U$ has image in a subset of $S^1$ homeomorphic to $J = (0,1)$, and so we may
	apply Lemma \ref{lem:MapsToInterval} to $G|_U$ to obtain a further homotopy of $G$ supported in $U$.  
	This has the effect of surgering $L$ along $\alpha$, which reduces of the number of components of $L$ by one 
	as shown in Figure \ref{fig:loop_join}.
	
	Another reduction of $L$ is possible if some component $\S_0$ of $\SmL$ has only one boundary 
	component.  In this case, we homotope $G$ by sending all of $\S_0$ past $q$; this homotopy can be taken
	to be the identity outside of any neighborhood of $\S_0$.  If $L'$ is the preimage of $q$ after
	the homotopy, then $L'$ consists of the same loops as $L$ except for the loop that formed the boundary of 
	$\S_0$, which has been eliminated.
	
	It follows that if $G$ is chosen to minimize the number of components of $L$, then every
	component $\S_0$ of $\SmL$ has exactly two boundary components: one co-oriented into $\S_0$ and 
	the other co-oriented out of $\S_0$, as shown in Figure \ref{fig:starfish_surface}.
	\begin{figure}[t]
		\begin{center}
			\begin{overpic}[width=0.7\textwidth]{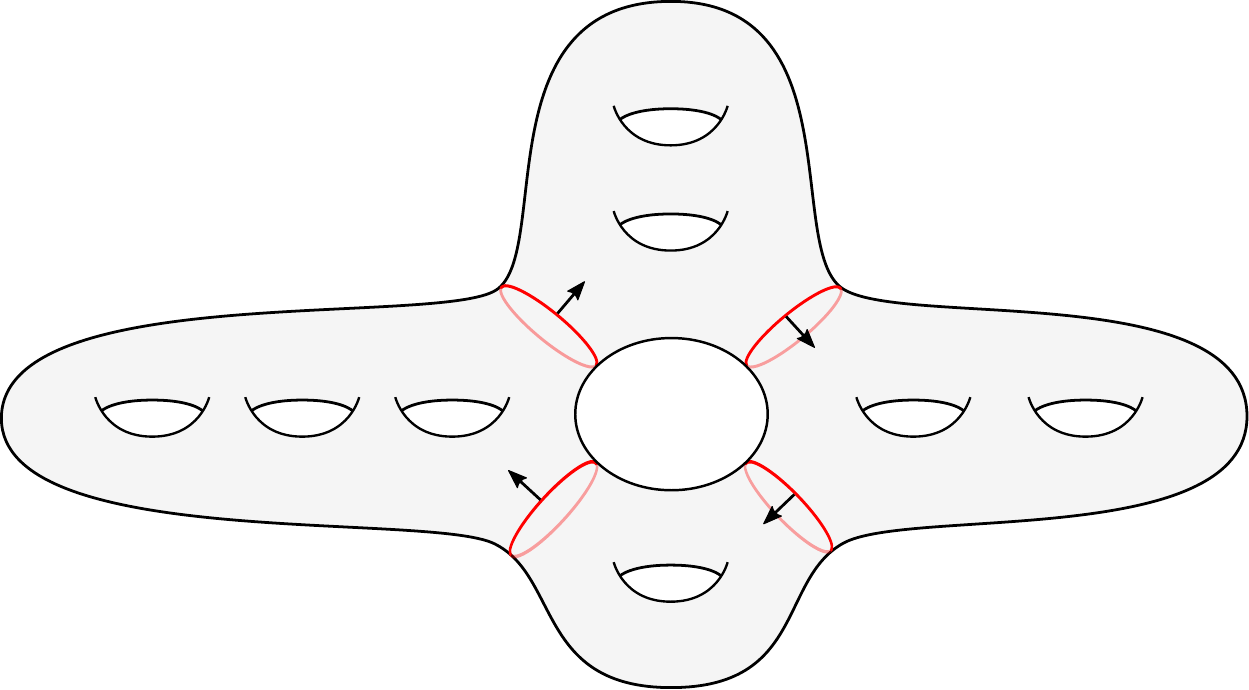}
			\end{overpic}
		\end{center}
		\caption{If $L$ has more than one component, then no loop in $\S$ can have a signed intersection
			of $\pm 1$ with $L$.}
		\label{fig:starfish_surface}
	\end{figure}
	We now observe that the homomorphism
	$G_* : \pi_1 \S \to \pi_1 S^1 \iso \Z$ is given by signed intersection with $L$, where the
	sign measures whether a loop in $\S$ agrees with the co-orientation of $L$.  From the construction
	of the co-orientation we see that
	$G_*$ must have image $\abs{L} \Z \le \Z$.  Since $G_*$ is surjective, we have
	$\abs{L} = 1$, and so $L$ is connected.  This completes the proof.
\end{proof}

\begin{proof}[Proof of Theorem \ref{thm:SLCforTBs_strong}]
	Let $P : M \to S^1$ denote the bundle projection of $M$, and let $G = P \circ F : \S \to S^1$. 
		
	\textsc{Case 1.} The map $G$ is $\pi_1$-surjective.
	Applying Lemma \ref{lem:MapsToCircle} to $G$, we may homotope $G$ so that the preimage of a 
	point $q \in S^1$ is a 2-sided simple loop $L \subset \S$ for which any loop in $\SmL$ has 
	inessential image under $G$.  Since we have that $G(\SmL) \subset \{S^1 \less q\}$,
	we may use the homotopy lifting property of the fiber bundle $M \to S^1$ to homotope 
	$F$ so that $F(\SmL) \subset M \less M_q$, where $M_q$ is the fiber of $M$ lying above $q$.
	
	Since $M \less M_q$ is homeomorphic to $T \times I$ and is therefore orientable,
	it follows from the 2-sidedness of $F$ that $\SmL$ 
	must be orientable.  Therefore $\SmL$ is an orientable compact surface with two boundary
	components, and so $\chi(\SmL) = \chi(\S)$ must be even.  This proves the claim that there is no
	2-sided map $\S \to M$ when $\chi(\S)$ is odd.  
	
	We may now suppose that $\chi(\S) = 2 - 2g$, where $g \ge 2$ is an integer.  Then $\chi(\SmL) = 2-2g$, so
	$\SmL$ is the connect sum of a twice-punctured sphere with $g-1$ tori.  It follows that there is
	an embedded punctured torus $\Sigma_0$ in $\SmL$. 
	The boundary loop $\beta$ of $\Sigma_0$ is a separating 
	simple loop in $\S$ whose corresponding element in $\pi_1 \S$ is the commutator of the elements represented by 
	loops $\gamma$ and $\delta$, as shown in Figure $\ref{fig:commutator_in_surface}$.
	\begin{figure}[t]
		\begin{center}
			\begin{overpic}[width=0.7\textwidth]{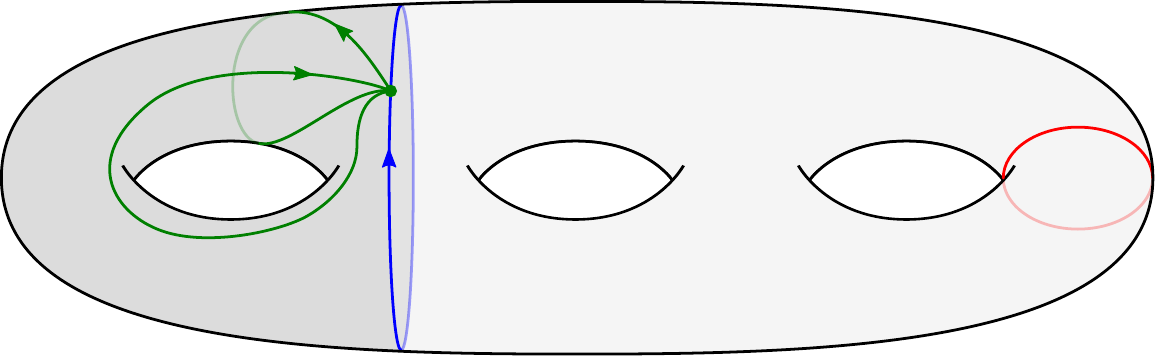}
	 			\put (25,31) {$\beta$}
	 			\put (15,07) {$\gamma$}
	 			\put (34,32) {$\delta$}
				\put (101,15) {$L$}
			\end{overpic}
		\end{center}
		\caption{The simple loop $\beta$ in $\ker F_*$ is the boundary of the punctured torus 
					$\Sigma_0 \subset \S$.}
		\label{fig:commutator_in_surface}
	\end{figure}
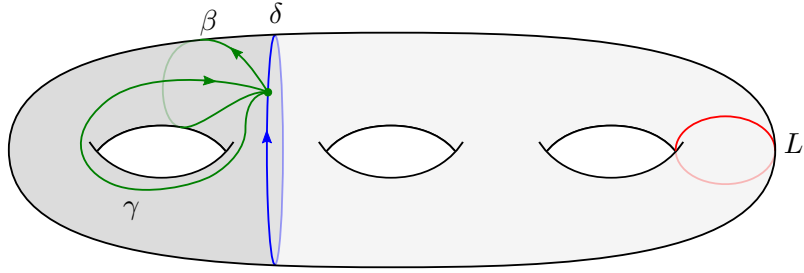
	The loops $\beta$, $\gamma$, and $\delta$ all have image in $M \less M_q$, and since $M \less M_q$ has
	abelian fundamental group it follows that $F_*[\beta]$ is trivial in $\pi_1 M$.  Thus $\beta$ is the desired 
	essential simple loop in the kernel of $F_*$.  (A similar argument shows that any essential separating loop in 
	$\SmL$ must represent an element of $\ker F_*$.)

	\textsc{Case 2.} The map $G$ is not $\pi_1$-surjective.
	In this case, either $G_*$ is the zero map or it has image $n \Z \le \Z \iso \pi_1 S^1$ for some 
	$n\ne 0, \pm 1$.  
	
	If $G_*$ is the zero map, then $G$ is homotopic to a constant map, and the homotopy can be 
	lifted to a homotopy of $F$ so that the resulting image of $\S$ is contained in a torus fiber $M_p$ of $M$.  
	Since $M_p$ is an orientable 2-sided submanifold of $M$, 
	by the 2-sidedness of $F$ we have that $\S$ is orientable, and so $\chi(\S)$ cannot be odd.  If
	$\chi(\S) \le -2$ then there is a essential separating loop in $\S$, and we argue as above that 
	such a loop represents an element of $\ker F_*$.
	
	If instead $G_*$ has image a finite index subgroup $n \Z \le \Z$, then $p_*\inv(n \Z)$ is  
	a proper finite-index subgroup of $\pi_1 M$ and $F$ lifts to the corresponding cover $\widetilde M \to M$.
	Since $\widetilde M$ must also be a torus bundle over a circle and
	the projection $\widetilde M \to M$ is $\pi_1$-injective, we may replace $M$ by $\widetilde M$ and $F$ by its 
	lift and appeal to Case 1.
\end{proof}

The following result will complete the proof of Theorem \ref{thm:SLCforTBs}.

\begin{prop}
	\label{prop:SLCforKleinbottles}
	Let $K$ be a Klein bottle and let $G$ be an infinite torsion-free group.
	If $f : \pi_1 K \to G$ is a homomorphism with nontrivial kernel, then there is a essential simple loop in $K$ 
	that represents an element of $\ker f$.
\end{prop}
\begin{proof}
	We proceed by reducing to the case in which $f$ has image an infinite cyclic subgroup of $G$.  Write the 
	fundamental group of $K$ as
	\[
		\pi_1 K = \ang{ a, b \div aba\inv b = 1},
	\]
	and let $H = \ang{ a^2, b } \le \pi_1 K$ be the index-2 subgroup of $\pi_1 K$ corresponding
	to the double cover of $K$ by a torus.  The kernel of $f|_H$ must be 
	nontrivial: for if $x \in \ker f_*$ is not the identity then $x^2 \in H \cap \ker f_*$ is also not the 
	identity.  Hence $f|_H$ is a non-injective map from a rank-2 free-abelian group to a torsion free group, and 
	so the image of $f|_H$ is either trivial or infinite cyclic.  If $f(H) = 1$, then since $f(a)^2 = f(a^2) = 1$
	and $M$ is torsion-free, $f(a)$ must be trivial. In this case $f$ is the trivial map and we're done.  
	If $f(H)$ is infinite cyclic, then $f(\pi_1 K)$ is a virtually-infinite-cyclic torsion-free group, and so must
	be infinite cyclic (see, for instance, \cite[Theorem 5.12]{ScottWall}).  
	
	Therefore we may replace $f$ by a surjective map $f' : \pi_1 K \to \Z$.  Since $S^1$ is a $K(\Z,1)$, there is 
	a map $F : K \to S^1$ with $F_* = f'$, and so Lemma \ref{lem:MapsToCircle} can be applied to obtain
	a essential 2-sided simple loop $L \subset K$ such that every loop in $K \less L$ has inessential image in 
	$S^1$.  Hence we see that $K \less L$ is an annulus, the core of which is a essential simple loop in $K$ that 
	represents an element of $\ker f'$, and hence of $\ker f$.
\end{proof}

\section{Torus Semi-Bundles}
	\label{sec:TSBs}
	
The following theorem, together with Theorem \ref{thm:SLCforTBs}, will establish Theorem \ref{thm:SLCforSolvmflds}.

\begin{theorem}
	\label{thm:SLCforTSBs}
	If $M$ is an orientable torus semi-bundle that admits a geometric structure modeled on $\sol$, then the Simple 
	Loop Conjecture holds for $M$.
\end{theorem}

As in the torus bundle case, we have a slightly stronger statement for maps from surfaces of sufficiently 
	large genus into orientable torus semi-bundles.

\begin{theorem}
	\label{thm:SLCforTSBs_strong}
	Let $\S$ be a closed surface and let $M$ be an orientable torus semi-bundle. If $\chi(\S) < -2$
	and $F : \S \to M$ is a 2-sided map, then there is an essential simple loop in $\S$ that represents an element 
	of $\ker F_*$.
\end{theorem}
	
To prove the theorem, we will employ the following two lemmas, which allow us to homotope maps from surfaces to 
	torus semi-bundles into a simplified position.  

\begin{lemma}
	\label{lem:TSBhomotopy}
	Let Let $M$ be an orientable torus semi-bundle with middle torus $S \subset M$, let $\S$ be a (not necessarily
	closed) surface, and let $F : \S \to M$ be a map that is transverse to $S$.  Suppose that $\alpha \subset \S$
	is a simple arc that connects two distinct components of $L = F\inv(S)$ whose interior is disjoint from $L$
	and that $F(\alpha)$ is homotopic (rel endpoints) into $S$.  
	Then $F$ can be homotoped in a neighborhood of $\alpha$ so that the preimage of $S$ changes by
	surgery along $\alpha$.
\end{lemma}

\begin{proof}
	Let $U$ be a tubular neighborhood of $\alpha$ in $\S$ that does not intersect any components of $L$ except 
	the two that are connected by $\alpha$.  Since $F(\alpha)$ is homotopic into $S$, after possibly shrinking $U$
	we can homotope $F$ with support in $U$ so that $F|_U$ has image that does not intersect either of the 1-sided 
	surfaces that are the zero sections of the twisted $I$-bundles that were used to construct $M$. 
	
	It follows that $F|_U$ has image in a subset of $M$ that is homeomorphic to $T \times J$, where $T$
	is a torus and $J = (0,1)$.  Let $P : T \times J \to J$ denote the projection onto the second factor, and 
	let $r \in J$ be the image of $S$.  Then $P \circ F|_U : U \to J$ satisfies the assumptions of 
	Lemma \ref{lem:MapsToInterval}, so we may apply it to obtain a homotopy of $P \circ F|_U$ after which
	$L$ has been surgered along $\alpha$.  Since $T \times J \to J$ is a fiber bundle, we can lift the homotopy
	of $P \circ F|_U$ to a homotopy of $F|_U$, and from that we obtain a homotopy of $F$ supported in $U$,
	as desired.
\end{proof}

\begin{lemma}
	\label{lem:MapsToTSBs}
	Let $M$ be an orientable torus semi-bundle with middle torus $S \subset M$, let $\S$ be a closed surface with  
	$\chi(\S) < 0$, and let $F : \S \to M$ be a (2-sided) map that injects on simple loops (that is, there are 
	no elements represented by simple loops in the kernel of $F_*$).  Then $F$ can be homotoped so that 
	$L = F\inv(S)$ is either empty or is a collection of parallel 2-sided separating essential simple loops in 
	$\S$.
\end{lemma} 
	
	Figure \ref{fig:baguette} shows a typical picture of $L \subset \S$ when $L \ne \emptyset$.
	\begin{figure}[t]
		\begin{center}
			\begin{overpic}[width=0.8\textwidth]{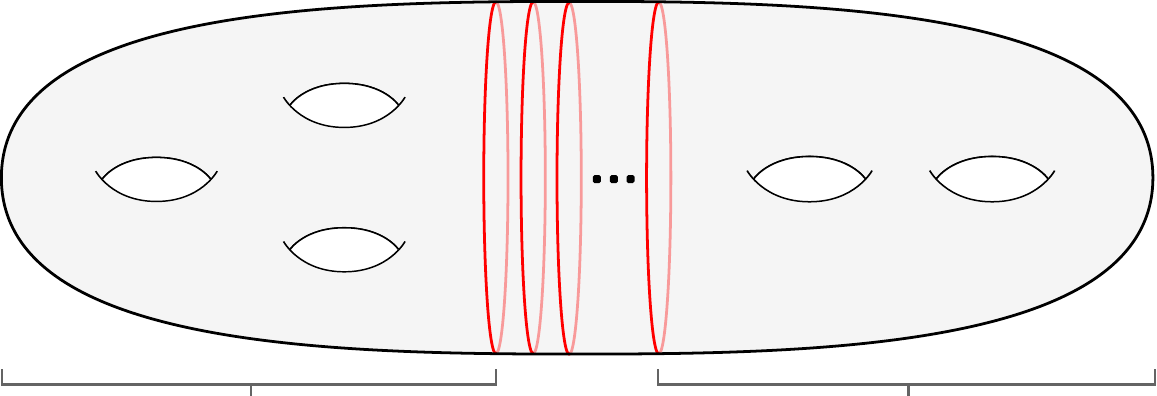}
				\put (50,0) {$L$}
				\put (20.5,-3.5) {$\S_1$}
				\put (77.5,-3.5) {$\S_2$}				
			\end{overpic}
		\end{center}
		\caption{The multicurve $L$ is a collection of parallel loops separating $\S$ into 
					a collection of annuli along with two punctured surfaces, $\S_1$ and $\S_2$.}
		\label{fig:baguette}
	\end{figure}
	
\begin{proof}
	In the notation of Section \ref{sec:TSBdefs}, let $M = N_1 \cup_\phi N_2$ with monodromies 
	$\rho_i : \pi_1 N_i \to \Z/2$.	
	Choose $F$ within its homotopy class so that $F$ is transverse to $S$ and so that $L = F\inv(S)$ is a 
	\emph{minimal} collection of 2-sided simple loops in $\S$.
	
	\textsc{Step 1.} 
	First, suppose that some component $\S_0$ of $\SmL$ has three or more boundary components.
	Let $C_1$, $C_2$, $C_3$ be three of the boundary components of $\S_0$.  
	(Since $S$ separates $M$, no two of the $C_i$ correspond to the same component of $L$.)
	Choose a basepoint $q \in S$; after a homotopy of $F$ supported in a tubular 
	neighborhood of the $C_i$, we may assume that each $C_i$ contains a point $p_i$ for which $F(p_i) = q$.  
	In $\S_0$ choose simple arcs $\alpha$ from 
	$p_1$ to $p_2$, $\alpha'$ from $p_2$ to $p_3$, and $\alpha''$ from $p_1$ to $p_3$ such that $\alpha''$ is
	path-homotopic to the concatenation of $\alpha$ and $\alpha'$, as shown in Figure 
	\ref{fig:three_arcs}.
	\begin{figure}[t]
		\begin{center}
			\begin{overpic}[width=0.7\textwidth]{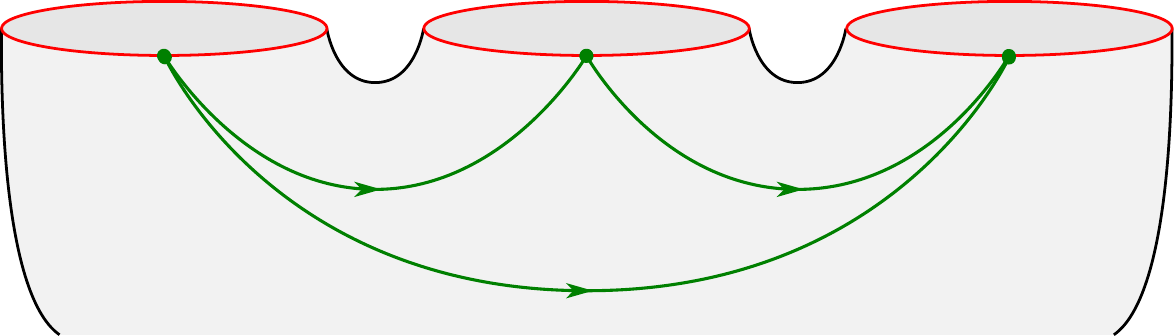}
				\put (02,30) {$C_1$}
				\put (38,30) {$C_2$}
				\put (74,30) {$C_3$}
				\put (13,25.5) {\small$p_1$}
				\put (49,25.5) {\small$p_2$}
				\put (85,25.5) {\small$p_3$}								
	 			\put (30,15) {$\alpha$}
	 			\put (66,15) {$\alpha'$}
	 			\put (48,7) {$\alpha''$}
				\put (100, 0) {$\S_0$}
			\end{overpic}
		\end{center}
		\caption{The arcs $\alpha$, $\alpha'$, and $\alpha''$ joining the boundary components of $\S_0$.}
		\label{fig:three_arcs}
	\end{figure}
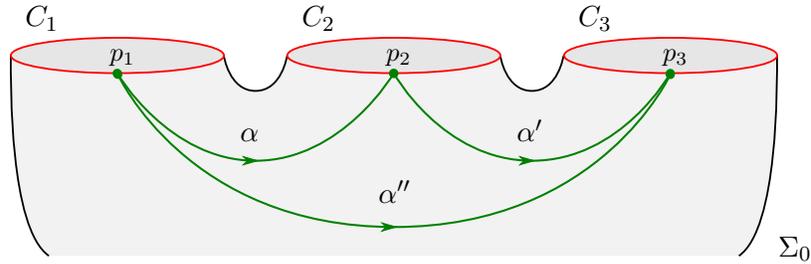
	By construction, each of $F(\alpha)$, $F(\alpha')$, and $F(\alpha'')$ are loops in $M$ based at $q$, 
	and without
	loss of generality all three lie in $N_1$.  It follows that $\rho_1[F(\alpha)]$, $\rho_1[F(\alpha')]$, and 
	$\rho_1[F(\alpha'')]$ are elements in $\Z/2$ with 
	$\rho_1[F(\alpha)] + \rho_1[F(\alpha')] = \rho_1[F(\alpha'')]$,
	and so one of the three elements must be trivial in $\Z/2$.  Hence one of the arcs (say $\alpha$) in 
	$\S_0$ has image under $F$ that is homotopic into $\bd N_1 = S$, and so by Lemma \ref{lem:TSBhomotopy} 
	we can homotope $F$ so that the result on $L$ is surgery along $\alpha$, which reduces the number 
	of components of $L$.
	
	\textsc{Step 2.} 
	Next, suppose that some component $\S_0$ of $\SmL$ has two boundary components and is not an annulus.
	As in the previous step, we can homotope $F$ in a neighborhood of $\bd \S_0$ so that each boundary
	component has a point $p_i$ ($i = 1,2$) that maps to the basepoint $q \in S$.  Without loss of generality
	we assume that $F(\S_0) \subset N_1$.  There are two cases to consider.
	
	\textsc{Case 2A.} There is a simple loop $\alpha \subset \S_0$ based at $p_1$ with $\rho_1[F(\alpha)]$ 
	nontrivial in $\Z/2$.  Homotope $\alpha$ in $\S_0$ so that $\alpha$ becomes the concatenation of two simple 
	arcs $\alpha'$ and $\alpha''$ from $p_1$ to $p_2$, as shown in Figure \ref{fig:arcs_in_tube1}.
	\begin{figure}[t]
		\begin{center}
			\begin{overpic}[width=0.36\textwidth]{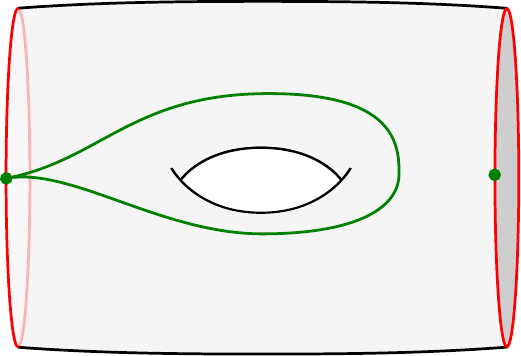}
				\put (-7,33) {\small$p_1$}
				\put (98,33) {\small$p_2$}							
	 			\put (50,18) {$\alpha$}
				\put (101, 0) {$\S_0$}
			\end{overpic}
			\SNAKE{1.5cm}
			\begin{overpic}[width=0.36\textwidth]{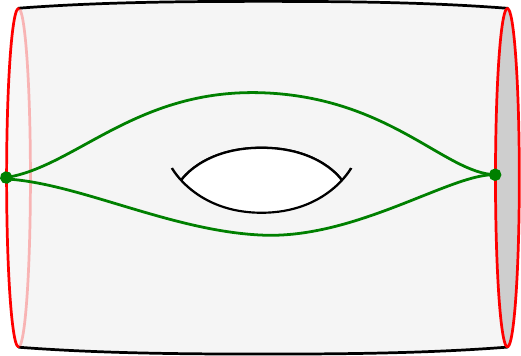}
				\put (-7,33) {\small$p_1$}
				\put (98,33) {\small$p_2$}							
	 			\put (45,53) {$\alpha'$}				
	 			\put (50,17) {$\alpha''$}
				\put (101, 0) {$\S_0$}
			\end{overpic}
		\end{center}
		\caption{Pulling $\alpha$ towards $p_2$ and viewing it as two arcs.}
		\label{fig:arcs_in_tube1}
	\end{figure}
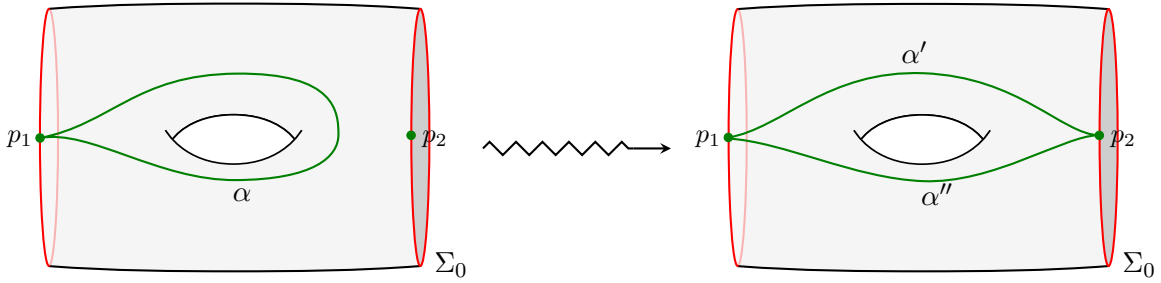
	It follows that $F(\alpha')$ and $F(\alpha'')$ are loops in $N_1$ based at $q$, and since
	$\rho_1[F(\alpha')] + \rho_1[F(\alpha'')] = \rho_1[F(\alpha)]$ is nontrivial in $\Z/2$, one of 
	$\rho_1[F(\alpha')]$ and $\rho_1[F(\alpha'')]$ must be trivial.  As before, an arc with trivial image can 
	be used (Lemma \ref{lem:TSBhomotopy}) to homotope $F$ surger $L$, which reduces the number of components of 
	$L$ by one.
	
	\textsc{Case 2B.} For every simple loop $\alpha \subset \S_0$ based at $p_1$, $\rho_1[F(\alpha)]$ is trivial.
	Since we assumed $\S_0$ is not an annulus, it is a twice-punctured orientable surface of genus 
	greater than 0.  It follows that we can find two simple loops $\gamma$ and $\delta$ in $\S_0$ whose commutator 
	in $\pi_1 \S_0$ is represented by a simple loop $\beta$; see Figure \ref{fig:arcs_in_tube2}.
	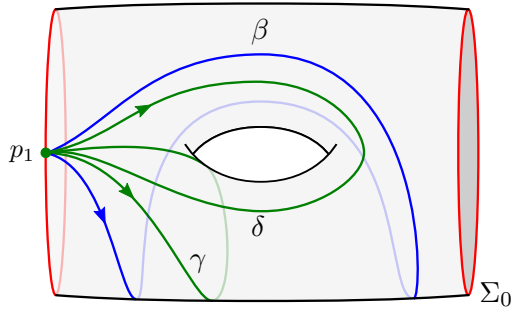
\begin{figure}[t]
		\begin{center}
			\begin{overpic}[width=0.4\textwidth]{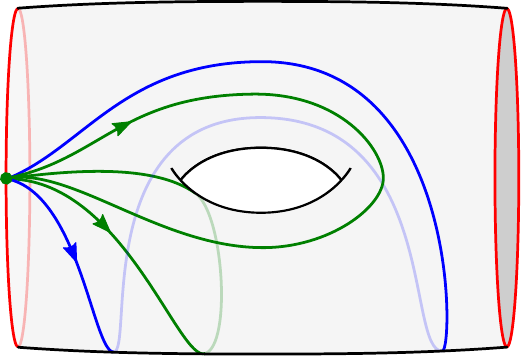}
				\put (-7,33) {\small$p_1$}
	 			\put (34,08) {$\gamma$}
	 			\put (48,15) {$\delta$}
	 			\put (48,60) {$\beta$}
				\put (100,0) {$\S_0$}
			\end{overpic}
		\end{center}
		\caption{The simple loop $\beta$ represents the commutator of $[\gamma]$ and $[\delta]$.}
		\label{fig:arcs_in_tube2}
	\end{figure}
	Since $[\beta], [\gamma], [\delta] \in \pi_1 \S_0$ all have trivial image under $\rho_1 \circ F_*$, 
	$\rho_1[F(\beta)]$, $\rho_1[F(\gamma)]$, and $\rho_1[F(\delta)]$ must lie in the subgroup of $\pi_1 N_1$ 
	corresponding to the boundary $S$.  But since $\pi_1S$ is abelian, the commutator $F_*[\beta]$ is trivial.  
	This contradicts the assumption that $F$ injects on simple loops, and so it is impossible that 
	$\rho_1 \circ F_*$ is trivial on every simple loop in $\S_0$.  
	
	We conclude that the number of components of 
	$L$ can be reduced whenever some component of  $\SmL$ has exactly two boundary components and is not 
	an annulus.
	
	\textsc{Step 3.}
	It follows from the previous two steps that if $F$ is chosen in its homotopy class so that $L$ has
	a minimal number of components, then $L$ is either empty or every component of $\SmL$ is either an
	annulus or a surface with exactly one boundary component.  The assumption that $\chi(\S) < 0$ rules out the 
	possibility that \textit{every} component of $\SmL$ is an annulus, and so $\S$ consists of two 
	punctured orientable surfaces connected by some number of annuli. 
\end{proof}

\begin{proof}[Proof of Theorem \ref{thm:SLCforTSBs_strong}]
	Let $\S$ be a closed surface with $\chi(\S) < -2$, let $M = N_1 \cup_\phi N_2$ be a torus semi-bundle, and let 
	$F: \S \to M$ be a 2-sided map.  By Lemma \ref{lem:MapsToTSBs}, we may assume that $F$ has been homotoped so 
	that $L = F\inv(S)$ is either empty or is a collection of parallel curves as in Figure 
	\ref{fig:baguette}.  (According to the lemma, if this is not possible then we can already find a 
	simple loop in $\ker F_*$.)  
	
	If $L = \emptyset$ then without loss of generality $F$ has image in $N_1$, which is homotopy equivalent
	to a Klein bottle.	Since $\pi_1 N_1$ does not contain the fundamental group of any surface of negative Euler
	characteristic, the induced map $\pi_1 \S \to \pi_1 N_1$ has nontrivial kernel.  Using Gabai's result
	\cite{Gabai1985}, we conclude that there is a simple loop in the 
	kernel of $F_*$.
	
	We now consider the case in which $L \ne \emptyset$.  If $\S_1$ and $\S_2$ are the two non-annular 
	subsurfaces of $\S$ as shown in Figure \ref{fig:baguette}, then 
	\[	 
		\chi(\S_1) + \chi(\S_2) = \chi(\S).
	\]
	It follows that either $\chi(\S_1) < -1$ or $\chi(\S_2) < -1$. 
   
	Without loss of generality, we will henceforth assume that $\chi(\S_1) < -1$ and
	that $F(\S_1) \subset N_1$.
	
	If $f = \rho_1 \circ (F|_{\S_1})_* : \pi_1(\S_1) \to \Z/2$, then since $F$ sends $\bd \S_1$ (which is a 
	component of $L$) into $S$, we have $f[\bd \S_1] = 0$.  It follows that $f$ represents a class in
	$H^1(\S_1, \bd\S_1; \Z/2)$.  If $f$ represents the trivial class, then all of $F(\S_1)$ is homotopic into 
	$S$, and we can homotope $F$ to send all of
	$\S_1$ past $S$ and reduce the number of components of $L$, contradicting the assumption that $F$ has already
	been homotoped to minimize the number of components. 
	Therefore $f$ is nontrivial in $H^1(\S_1, \bd\S_1; \Z/2)$, and so by Lefschetz Duality, there is a nontrivial 
	homology class $f_* \in H_1(\S_1; \Z/2)$ for 
	which the value of $f$ on any loop $\alpha$ based on $\bd \S_1$ is given by the signed intersection 
	($\bmod\ 2$) of $\alpha$ with any $1$-chain representing of $f_*$.  
	
	Let $\ell$ be a simple loop in $\S_1$
	that represents $f_*$. (A simple loop representative exists by \cite{Meyerson1976}.)
	Since $f_*$ is nontrivial, $\ell$ is essential and every loop in $\S_1 \less \ell$ 
	is in the kernel of $f$ and therefore has image in $N_1$ that is homotopic into $S$.  
	The fact that $\chi(\S_1) < -1$ implies that $\S_1 \mless \ell$ is homeomorphic to 
	a closed surface of genus at least one with three open discs removed.  
	As in the proof of 
	Theorem \ref{thm:SLCforTBs_strong}, we can find an embedded punctured torus $P$ in $\S_1 \mless \ell$ whose
	boundary $\beta$ represents the commutator of simple loops $\gamma$ and $\delta$ contained in $P$.  Since
	$[\beta]$, $[\gamma]$, and $[\delta]$ all have image under $F_*$ in the abelian subgroup 
	$\pi_1 S \le \pi_1 M$, we conclude that $\beta$ is the desired simple loop representing an element of 
	$\ker F_*$.
\end{proof}
	
With Proposition \ref{prop:SLCforKleinbottles} and the proof of the Simple Loop Conjecture when the domain
	is a torus given in \cite{Hass1987}, we will complete the proof of Theorem \ref{thm:SLCforTSBs} with the
	following special case.
	
\begin{lemma}
	\label{lem:SLCforGenus2intoTSB}
	Let $\S$ denote the closed orientable surface with $\chi(\S) = -2$.  If $M$ is an orientable torus semi-bundle
	and $F:\S \to M$ is a (2-sided) map, then either there is a essential simple loop in $\ker F_*$ or $M$ does
	not admit a geometric structure modeled on $\sol$.
\end{lemma}

\begin{proof}
	By Lemma \ref{lem:MapsToTSBs}, we can homotope $F$ so that the preimage $L = F\inv(S)$
	of the middle torus of $M$ is a minimal collection of parallel curves in $\S$ as in Figure 
	\ref{fig:baguette}.  As in the proof of Theorem \ref{thm:SLCforTSBs_strong} we may also 
	assume that $L \ne \emptyset$, so $L$ separates $\S$ into punctured tori $\S_1$ and $\S_2$ along with a 
	collection of $n = \abs{L}-1$ annuli.
	
	\textsc{Case 1}: $n = 0$.
	In this case, $L$ is connected and separates $\S$ into punctured tori 
	$\S_1$ and $\S_2$.  We can write the fundamental group of $\S$ as 
	\[
		\pi_1 \S = \ang{a_1, b_1, a_2, b_2 \mid [a_1,b_1] = [a_2,b_2]},
	\]
	where $a_i$ and $b_i$ are the generators of the fundamental group of $\S_i$.  The fundamental group of $M$
	has presentation
	\[
		\pi_1 M = \ang{x_1, y_1, x_2, y_2 \mid x_i y_i x_i\inv y_i = 1, 
				x_1^2 = x_2^{2r} y_2^t, y_1 = x_2^{2s} y_2^u },
	\]
	where $x_i$ and $y_i$ are the generators of the fundamental group of the twisted $I$-bundle over a Klein 
	bottle $N_i$, and $M$ has been constructed by gluing $N_1$ to $N_2$ via a homeomorphism $\bd N_1 \to \bd N_2$
	whose matrix is 
	\[
		\begin{pmatrix} r & s \\ t & u \end{pmatrix} \in GL_2(\Z)
	\]
	with respect to the bases $\ang{x_i^2, y_i}$ of the fundamental groups of the boundaries of the $N_i$.  
	By the definition of $L$ we see that $F$ restricts to a proper map of $\S_i$ into $N_i$, and so $F_*(a_i)$ and 
	$F_*(b_i)$ must lie in $\ang{x_i, y_i}$ for $i = 1,2$.  The subgroup $\ang{x_i,y_i}$ of $\pi_1M$ is isomorphic 
	to the fundamental group of a Klein bottle, and its commutator subgroup is infinite cyclic with generator 
	$y_i^2$.  Hence the commutators $[a_i,b_i]$ are mapped to even powers of $y_i$, and from the relation in 
	$\pi_1 \S$ we obtain an equation
	\[
		y_1^{2k_1} = y_2^{2k_2}
	\]
	for some integers $k_1$ and $k_2$.
	Applying the rightmost relation of the presentation of $\pi_1M$ given above, we have
	\[
		x_2^{4sk_1} y_2^{2uk_1} = y_2^{2k_2}.
	\]
	Since this is an equation in $\ang{x_2^2, y_2} \iso \Z \oplus \Z$, we can conclude that $4sk_1 = 0$,
	and so either $k_1 =0$ or $s=0$.  If $k_1 = 0$, it follows that the curve $L$
	(which represents the elements $[a_1,b_1]$ and $[a_2,b_2]$ in $\pi_1\S$) has image $y_1^{2k_1} = 1$,
	so $L$ is a essential simple loop in the kernel of $F_*$.  If $s =0$, then by Theorem 
	\ref{thm:solvmfldClassification} it follows that $M$ does not admit a geometric structure modeled on $\sol$.
	
	\textsc{Case 2}: $n > 0$.
	In this case, $L$ has multiple components; we will show that $F$ can be lifted to a torus semi-bundle cover of 
	$M$ in which the preimage of the middle torus is connected, thereby reducing to the case in which $n = 0$.
	Choose points $p_0, \ldots, p_{n}$ on the $n+1$ components of $L$, and let $\alpha \subset \S$ be a simple
	arc with end points at $p_0$ and $p_n$ whose intersection with $L$ is the points $p_i$.  For 
	$i = 0, \ldots, n-1$
	let $\alpha_i$ denote the segment of $\alpha$ between $p_{i}$ and $p_{i+1}$, as shown in Figure
	\ref{fig:baguette_special}.
	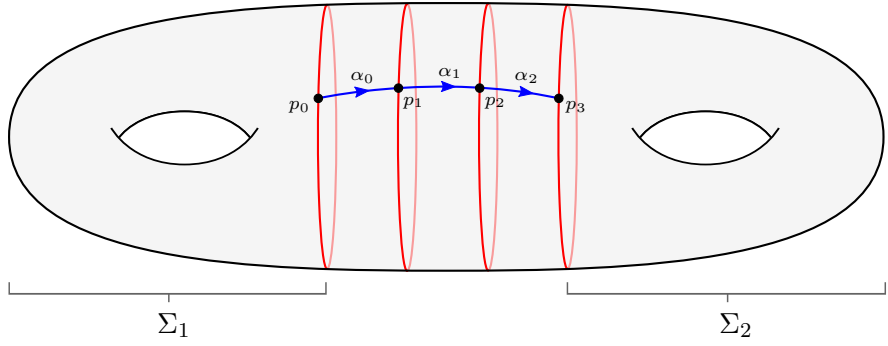
\begin{figure}[t]
		\begin{center}
			\begin{overpic}[width=0.8\textwidth]{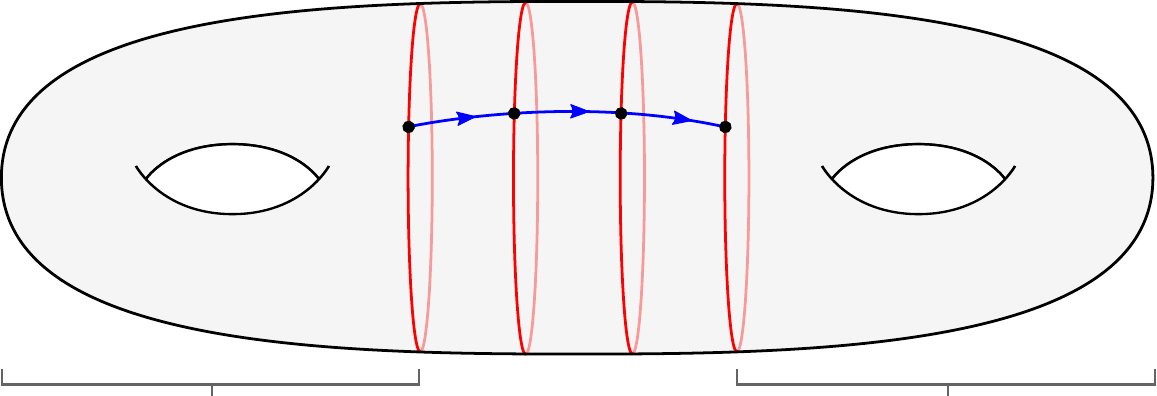}
				\put (32,22)   {\tiny $p_0$}
				\put (45,22.5)   {\tiny $p_1$}
				\put (54.3,22.5)   {\tiny $p_2$}
				\put (63.5,22)   {\tiny $p_3$}
				
				\put (39,25.3)   {\tiny $\alpha_0$}
				\put (49,25.8)   {\tiny $\alpha_1$}
				\put (57.7,25.3)   {\tiny $\alpha_2$}
				
				\put (17,-3.5) {$\S_1$}
				\put (81,-3.5) {$\S_2$}				
			\end{overpic}
		\end{center}
		\caption{The arc $\alpha$ connecting the points $p_i$ in the case $n = 3$.}
		\label{fig:baguette_special}
	\end{figure}
	By adjusting $F$ by a homotopy that preserves $L$, we may assume that $F(p_i) = q$ for some basepoint
	$q \in S \subset M$, and so $F(\alpha_i)$ is a loop in $M$ based at $q$ representing an element 
	$w_i \in \pi_1M$.  
	
	In the notation of the previous case, we assume that $F_*(a_1)$ and $F_*(b_1)$ lie in 
	$\ang{x_1, y_1} \le \pi_1M$, and by the definition of $L$ we have that $w_i \in \ang{x_{j_i},y_{j_i}}$ where
	$j_i = 1$ if $i$ is odd and $j_i = 2$ if $i$ even.  We may also assume that 
	$w_i \notin \ang{x_{j_i}^2,y_{j_i}}$, for if $w_i \in \ang{x_{j_i}^2,y_{j_i}}$ then $\alpha_i$ is a proper 
	simple arc in a component $\SmL$ with image homotopic into $S$, and we can reduce the number of 
	components of $L$, which contradicts the minimality assumption.  If $w = w_0\cdots w_{n-1}$, then we have
	\[
		F_*(\pi_1\S) \le \ang{x_1, y_1, w x_k w\inv, w y_k w\inv},
	\]
	where $k = 1$ if $n$ is odd and $k = 2$ if $n$ is even.  
	
	If $D = \ang{g_1, g_2 \mid g_1^2 = g_2^2 = 1}$ denotes the infinite dihedral group, then there is a 
	homomorphism $f: \pi_1 M \to D$ given by $x_i \mapsto g_i$ and $y_i \mapsto 1$ for $i = 1,2$.  The cover of 
	$M$ corresponding to $\ker f$ is $T \times \R$ with deck group $D$,
	as described in Section \ref{sec:TSBdefs}.  For 
	each $i = 0, \ldots, n-1$, since $w_i \notin \ang{x_{j_i}^2, y_{j_i}}$ we have $f(w_i) = g_{j_i}$, and it 
	follows that $f(w)$ is a reduced word in $D$ of length $n$ starting with $g_2$.  The image of $\pi_1 \S$ under
	the composition $f \circ F_*$ is the subgroup
	\[
		H = \ang{g_1, f(w) g_k f(w)\inv} \le D,
	\]
	which itself is isomorphic to the infinite dihedral group.  
	Let $\hat M$ be the quotient of $S \times \R$ by $H$, which is another torus semi-bundle that is
	the cover of $M$ corresponding to the subgroup $f\inv(H)$.
%
	Then $\hat M$ contains $n+1$ tori $S_0, \ldots, S_n$ that are lifts of $S$, and 
	the result of splitting $\hat M$ along these tori is $n$ 
	products $T \times I$ (each of which double-covers $N_1$ or $N_2$) along with two twisted $I$-bundles 
	over a  Klein bottle (each of which projects to $N_1$ or $N_2$ by a homeomorphism).
%
	The $S_i$ are parallel and one can show that $\hat F\inv(S_i)$ is connected for $i = 0, \ldots, n$,
	where $\hat F : \S \to \hat M$ is the lift of $F$ to $\hat M$.  
	Hence we can take any of the $S_i$ to be the ``middle torus'' of $\hat M$.

	Therefore we may apply the argument of the first case of this proof to $\hat F$ to find either a essential 
	simple loop in $\ker \hat{F}_*$ or that $\hat M$ is Seifert fibered.  In the former case, an essential simple 
	loop in $\ker \hat F_*$ is also an essential simple loop in $\ker F_*$. In the latter, if $\hat M$ is Seifert 
	fibered then it carries a Euclidean or \textit{Nil} structure, and therefore so does $M$.  It follows that $M$ 
	is Seifert fibered as well.	
\end{proof}

\section{The Simple Loop Conjecture for Metabelian Groups}
	\label{sec:SLCforMetabelianGroups}

	An \emph{orientation character} on a group $G$ is a homomorphism $\rho_G : G \to \Z/2$, and an 
	\emph{oriented group} is a pair $(G, \rho_G)$ where $\rho_G$ is an orientation on $G$.  When $G$ is the 
	fundamental group of a manifold $M$, we take $\rho_G$ to be the orientation character $\rho_M$ defined 
	in Section \ref{sec:defs}.  Similarly, one can say what it means for a homomorphism between two oriented
	groups to be \emph{2-sided}.  It then seems natural to ask if the following
	generalization of the Simple Loop Conjecture holds for a fixed oriented group $G$.
	
	\begin{statement*}
		Let $\S$ be a closed surface and let $(G, \rho_G)$ be an oriented group.  
		If $f : \pi_1\S \to G$ is a 2-sided homomorphism that is not 
		injective, then there is an essential simple loop in $\S$ that represents an element of the kernel of $f$.
	\end{statement*}
	
When $G$ is the fundamental group of an aspherical 3-manifold this is equivalent to 
	the Simple Loop Conjecture for 3-manifolds.  This statement is known to be false when $G = \mbox{PSL}(2,\C)$
	by work of Cooper-Manning \cite{Cooper2011} and when $G = \mbox{PSL}(2,\R)$ by work of Mann \cite{Mann2014}.
	(In both cases, $G$ carries the trivial orientation character as it is identified with the groups
	of orientation-preserving isometries of hyperbolic 3- and 2-space, respectively.)
	
A group is called \emph{metabelian} if it fits into a short exact sequence of the form
	\[
		1 \longrightarrow A \longrightarrow G \longrightarrow B \longrightarrow 1,
	\] 
	where $A$ and $B$ are abelian groups.  For example, the fundamental groups of the torus bundles treated in 
	Section \ref{sec:TBs} are metabelian with $A = \Z \oplus \Z$ and $B = \Z$.  One might be led to ask
	if the group-theoretic version of the Simple Loop Conjecture holds for metabelian groups, and if a 
	technique similar to that of Section \ref{sec:TBs} can be used to prove it.  We provide the following
	result in this direction.
	
\begin{theorem}
	\label{thm:SLCforZextensions}
	Let $(G,\rho_G)$ be an oriented group that fits into an exact sequence of the form
	\[
		1 \longrightarrow A \longrightarrow G \longrightarrow \Z \longrightarrow 1,
	\]	
	where $A$ is abelian, and suppose that $A \le \ker \rho_G$.
	If $\S$ is a closed surface of genus at least two, then the group-theoretic version of the Simple Loop 
	Conjecture holds for $\S$ and $G$.  
\end{theorem}

\begin{proof}
	This will be a group-theoretic analogue to the proof of Theorem \ref{thm:SLCforTBs_strong}.  
	Let $p : G \to \Z$ denote the projection map in the short exact sequence.  For a surface $\S$ and a 2-sided 
	homomorphism $f: \pi_1 \S \to G$, we may assume that $f$ is surjective.  For if not, then either
	$f(\pi_1 \S)$ lies in $A$ and any separating simple loop in $\S$ represents an element of $\ker f$, 
	or $p \circ f$ has nontrivial image and we replace $G$ by 
	$f(\pi_1 \S)$, $\rho_G$ by $(\rho_G)|_{f(\pi_1 \S)}$, $A$ by $A \cap f(\pi_1 \S)$, 
	and $\Z$ by $(p \circ f)(\pi_1 \S) \iso \Z$.
	
	There is a map $\S \to S^1$ whose induced homomorphism on fundamental groups is $p \circ f$, and by applying 
	Lemma \ref{lem:MapsToCircle} to this map we find a simple nonseparating loop $L \subset \S$ such that every 
	element of $\pi_1(\SmL) \le \pi_1 \S$ is contained in $\ker (p \circ f)$.  By exactness, 
	$f(\pi_1(\SmL))$ is contained in $A$, and the assumptions that $f$ is 2-sided and that 
	$A \le \ker \rho_G$ imply that $\SmL$ must be orientable.
	
	As shown in the proof of Theorem \ref{thm:SLCforTBs_strong} there are essential simple loops $\beta$, 
	$\gamma$, and $\delta$ 
	in $\S$ representing elements of $\ker(p \circ f)$ and with $[\beta]$ equal to the commutator of $[\gamma]$ 
	and $[\delta]$.  By exactness, $f[\beta]$, $f[\gamma]$, and $f[\delta]$ are contained in $A$, and since $A$ is 
	abelian we have that $f[\gamma]$ is trivial.
\end{proof}

We conclude by showing that, despite the previous result, the group-theoretic Simple Loop Conjecture does not
	hold for \textit{all} torsion-free metabelian groups.  This is a torsion-free version of a finite example
	due to Casson \cite[Section 2]{Livingston2000}.
	
\begin{example}
	\label{ex:group_slc_counterex}
	\newcommand{\tS}{\hat\S}
	\newcommand{\tbeta}{{\hat\beta}}
	Let $\S$ be a surface of genus $g \ge 2$.  We will give a topological construction of the quotient of 
	$\pi_1\S$ by its second derived subgroup, which is sometimes called the \textit{metabelianization} of 
	$\pi_1\S$.  From the construction we will see that the kernel of $\pi_1 \S \to G$ does not contain
	any elements represented by simple loops in $\S$.
	
	First, let $B = H_1(\S)$ (with $\Z$ coefficients understood), 
	let $f_1 : \pi_1 \S \to B$ be the abelianization map, and let 
	$K_1 = \ker f_1$.  Let $P : \tS \to \S$ be the cover of $\S$ corresponding to $K_1$.  Next, let 
	$f_2 : \pi_1 \tS \to H_1( \tS )$ 
	be the analogous natural map for $\tS$, and let $K_2 = \ker f_2$.  We have $K_2 \le \pi_1 \tS \iso K_1 \le 
	\pi_1 \S$, and so we identify $K_2$ with its image under $P_*$ and consider it a subgroup of $\pi_1 \S$.  
		
	Observe that $K_1$ does not contain any element of $\pi_1\S$
	represented by a nonseparating simple loop in $\S$, but does 
	contain every element represented by a separating simple loop in $\S$.  Hence every separating simple loop in 
	$\S$ lifts to $\tS$; we now show that every such loop lifts to a \textit{nonseparating} simple loop in $\tS$.  

	We first observe that $B \iso \Z^{2g}$ is a one-ended group.  Since $B$ acts properly on $\tS$
	with compact quotient $\S$, it follows that 
	$\tS$ is a one-ended space.  Any inessential separating simple loop in $\tS$ must therefore separate $\tS$ 
	into a compact piece and a noncompact piece.  Hence if $\beta$ is a simple separating loop 
	in $\S$ for which some (and hence any) lift $\tbeta$ of $\beta$ separates $\tS$, then $\tbeta$ cuts off a 
	compact subsurface $\tS_{\tbeta} \subset \tS$.  If $\tbeta'$ is another lift of $\beta$, then $\tbeta$
	and $\tbeta'$ are disjoint, and the regularity of the cover $\tS \to \S$ implies that there is a deck 
	transformation of $\tS$ that takes $\tbeta'$ to $\tbeta$.  This deck transformation must take $\tS_{\tbeta'}$ 
	homeomorphically onto $\tS_{\tbeta}$.  If one of these subsurfaces is contained in the other (say
	$\tS_{\tbeta'} \subset \tS_{\tbeta}$) then $\tbeta$ and $\tbeta'$ must be parallel.  However, this is 
	impossible: for by choosing hyperbolic metrics on $\S$ and $\tS$ so that the covering action is by isometries,
	and choosing $\beta$, $\tbeta$, and  $\tbeta'$ to be the unique geodesics in their homotopy classes, we 
	see that if $\tbeta$ and $\tbeta'$ are parallel then they are not distinct lifts of $\beta$.
	
	It follows that the 
	subsurfaces $\tS_\tbeta$ (as $\tbeta$ ranges over the lifts of $\beta$) must be disjoint.  In particular, 
	each such subsurface does not contain any lifts of $\beta$ in its interior.  Thus the covering map 
	$\tS \to \S$ restricts to a cover of a component of $\S \less \beta$ by $\tS_\tbeta$, and since $\tbeta$ 
	projects to 
	$\beta$ via a homeomorphism, the restricted cover is a homeomorphism.  However, this is 
	impossible, as $\tS_\tbeta$ is not a disk and so must contain a nonseparating simple loop, and this 
	nonseparating loop is a lift of its image under the covering projection.  We have already observed that such loops do not lift from $\S$ to $\tS$, and so from this contradiction we conclude 
	that $\tbeta$ (and hence every lift of $\beta$ to $\tS$) must be nonseparating.
	
	It follows that $K_2$ does not contain \textit{any} elements represented by simple loops of $\S$, since 
	the nonseparating simple loops in $\S$ are homologically nontrivial, and the separating simple loops of $\S$ 
	lift to homologically nontrivial loops in $\tS$.  Hence if we let $G = \pi_1 \S / K_2$ and let
	$f : \pi_1 \S \to G$ be the quotient map, then $f$ is a noninjective map with no elements represented by 
	essential simple loops in its kernel.  If $A = \pi_1 \tS / K_2 \iso H_1(\tS)$, then $A$ is abelian and 
	we have
	\[
		G/A = (\pi_1 \S / K_2)/(\pi_1 \tS / K_2) \iso \pi_1\S / \pi_1 \tS \iso \pi_1 \S / K_1 
			\iso H_1(\S ),
	\]
	which is also abelian.  Thus we see that $G$ is metabelian, for it fits into the short exact sequence
	\[
		1 \longrightarrow  H_1(\tS ) 
			\longrightarrow G 
			\longrightarrow H_1(\S ) 
			\longrightarrow 1,
	\]
	and so we have constructed the desired group $G$ and map $f: \pi_1 \S \to G$.
\end{example}

\bibliographystyle{siam}

\end{document}